\documentclass[reqno,draft]{amsart}
\usepackage{amsmath, amsthm, amssymb, dsfont}
\usepackage[a4paper]{geometry}					
\usepackage{graphicx}
\usepackage{tikz}
\usepackage{enumerate}

\usepackage{latexsym}
\usepackage{bbm}
\usepackage{mathtools}
\usetikzlibrary{shapes,snakes}
\usetikzlibrary{positioning}

\usepackage{comma}
\usepackage{cutwin}
\usepackage{pstricks}
\usepackage{pst-plot}
\usepackage{pgfplots}
\usepackage{float}

\numberwithin{equation}{section}

\usepackage{changes}

\theoremstyle{plain}
\newtheorem{theorem}{Theorem}[section]
\newtheorem{lemma}[theorem]{Lemma}

\theoremstyle{remark}
\newtheorem{remark}[theorem]{Remark}
\theoremstyle{definition}

\newcommand{\N}{\mathbb{N}}
\newcommand{\Z}{\mathbb{Z}}

\newcommand{\R}{\mathbb{R}}

\newcommand{\B}{\mathcal{B}}

\newcommand{\Prob}{\mathbf{P}}

\newcommand{\E}{\mathbf{E}}

\newcommand{\eqdist}{%
  \mathrel{\vbox{\offinterlineskip\ialign{%
    \hfil##\hfil\cr
    $\scriptscriptstyle\mathrm{d}$\cr
    \noalign{\kern.1ex}
    $=$\cr
}}}}

\newcommand{\1}{\mathbbm{1}}

\newcommand{\distto}{%
  \mathrel{\vbox{\offinterlineskip\ialign{%
    \hfil##\hfil\cr
    $\scriptscriptstyle\mathrm{d}$\cr
    \noalign{\kern-.05ex}
    $\to$\cr
}}}}
\newcommand{\Probto}{%
  \mathrel{\vbox{\offinterlineskip\ialign{%
    \hfil##\hfil\cr
    $\scriptscriptstyle\Prob$\cr
    \noalign{\kern-.05ex}
    $\to$\cr
}}}}
\newcommand{\TVto}{%
  \mathrel{\vbox{\offinterlineskip\ialign{%
    \hfil##\hfil\cr
    $\scriptscriptstyle\mathrm{TV}$\cr
    \noalign{\kern-.05ex}
    $\to$\cr
}}}}

\newcommand{\e}{\mathbf{e}}

\newcommand{\bU}{\boldsymbol U}

\newcommand{\defeq}{\vcentcolon=}
\newcommand{\eqdef}{=\vcentcolon}

\newcommand{\vel}{\overline{\mathrm{v}}}
\newcommand{\pesc}{p_{\mathrm{esc}}}
\newcommand{\lambdacrit}{\mathit{\lambda}_{\mathrm{crit}}}

\begin{document}
\title{Speed function for biased random walks with traps}
\author{Volker Betz \and Matthias Meiners \and Ivana Tomic}

\begin{abstract}
We consider a biased nearest-neighbor random walk on $\Z$
which at each step is trapped for some random time with random, site-dependent mean.
We derive a simple formula for the speed function in terms of the model parameters.
\smallskip

\noindent
{\bf Keywords:} Bouchaud's trap model, random walk in random environment, rate of escape, speed function
\\{\bf Subclass:} MSC: 60K37 $\cdot$ 60F15 
\end{abstract}

\maketitle

\section{Introduction}

Biased random walks on random graphs have been popular objects of study in 
recent years. The most relevant effect is that the random walk can run into traps, i.\,e.\ portions of the random graph that 
can only be exited by making a 
large number of steps in the direction opposite to the drift, which can take a long time. This leads to a decrease of the speed of the 
ballistic motion that is typical for a biased random in the absence of traps. Explicitly, if we embed the random graph into $\R^d$
and define the ballistic speed of the walk 
$(X_t)_{t \geq 0}$ as $\vel = \lim_{t \to \infty} \frac{1}{t} X_t$ (provided the limit exists almost surely),
then in many models $\vel$ (or rather its projection in the direction of the bias)
is not a monotone function of the bias, and often equals zero beyond a critical value of the bias. The latter 
behaviour is known to hold for the biased walk on the infinite connected cluster of supercritical bond percolation on $\Z^d$, $d \geq 2$ 
\cite{Berger+Gantert+Peres:2003,Fribergh+Hammond:2014,Sznitman:2003},
on the Galton-Watson tree with leaves 
\cite{Aidekon:2014,Bowditch+Tokushige:2020,Lyons+Pemantle+Peres:1996}
and on certain one-dimenisional percolation models 
\cite{A-F+H:2009,Gantert+al:2018}. 
We refer to the lecture notes \cite{BenArous+Fribergh:2016} and the introduction of \cite{Gantert+al:2018} for further details on the models 
and relations to physical models. 

The purpose of this note is to derive an explicit formula for the speed in a class of toy models that includes 
the biased Bouchaud's trap model on the integers. In these models, a random walk $(Y_n)_{n \in \N_0}$ runs on $\Z$, 
and the geometry of the random graph is replaced by a family of random average holding times. Explicitly, a random average 
holding time $w_x$ is sampled for each $x \in \Z$, and when the biased random walk hits the point $x$ at time $k$, 
it waits for a random time $\tau_{w_x,x,k}$, mimicking the time it takes to leave a trap with entrance at $x$. The stochastic processes
$(\tau_{r,x,k})_{r \geq 0}$, $x \in \Z$, $k \in \N_0$ are i.\,i.\,d.\ (independent and identically distributed)
and $\tau_{r,x,k}$ has mean $r$, but the $w_x$, $x \in \Z$ are not re-sampled during the dynamics, 
and thus the waiting times  $\tau_{w_{Y_k},Y_k,k}$, $k \in \N_0$ are not independent.
This means that a trivial application of the law of large numbers 
to compute the average waiting time per site, $\lim_{n \to \infty} \frac{1}{n} \sum_{k=0}^{n-1} \tau_{w_{Y_k},Y_k,k}$, 
is not possible. Under the assumption that $(w_x)_{x \in \Z}$ is an ergodic family, we will show that the law of large numbers nevertheless 
holds, and derive an explicit expression for $\vel$ in terms of the waiting time distributions. We illustrate our results with 
several examples that also demonstrate how close our toy model is to various models of biased walks on random graphs.

\section{Model, result and examples}		\label{sec:results}

For $\lambda > 0$, let $\xi, \xi_1, \xi_2,\ldots$ be a sequence of i.\,i.\,d.\ random variables, with
\begin{equation}	\label{eq:xi}
\Prob(\xi = 1) = 1-\Prob(\xi=-1) = \frac{{\rm e}^\lambda}{{\rm e}^\lambda + {\rm e}^{-\lambda}} \eqdef p_\lambda.
\end{equation}
We define $Y_n \defeq \sum_{k=1}^n \xi_k$ for $n \in \N_0 = \{0,1,2,\ldots\}$.
The process $(Y_n)_{n \in \N_0}$ is a biased nearest-neighbor random walk on $\Z$ starting at $0$.

Let $(w_x)_{x \in \Z}$ be an ergodic sequence of positive random variables, modeling the average waiting times at the points $x \in \Z$.
Let $((\tau_{r,x,k})_{r \geq 0})_{x \in \Z, k \in \N_0}$ be a family of independent random processes with $\E[\tau_{r,x,k}] = r$ for all $r \geq 0$
and all $x \in \Z$, $k \in \N$. We assume that $(\tau_{r,x,k})_{r \geq 0}$ for fixed $x \in \Z$ and $k \in \N_0$
is a stochastic process taking values in the Skorohod space $D = D[0,\infty)$ of right-continuous functions with existing left limits at all points in $(0,\infty)$,
see e.\,g.\ \cite[Section 16]{Billingsley:1999}.
Define $(X_t)_{t \geq 0}$ to be the continuous-time process that follows the trajectory of the walk $(Y_n)_{n \in \N_0}$
but following its $k^{\mathrm{th}}$ step to site $x \in \Z$, say, spends the random time $\tau_{w_x,x,k}$ at $x$
before making the next step.
More precisely, let $T_0 \defeq 0$ and
\begin{equation}	\label{eq:T_n}
T_n \defeq \sum_{k=0}^{n-1} \tau_{w_{Y_k}, Y_k,k} 
\end{equation}
for $n \in \N$. Then we define $(X_t)_{t \geq 0}$ via
\begin{equation}	\label{eq:BTM}
X_t \defeq Y_k	\quad	\text{for } T_k \leq t < T_{k+1}.
\end{equation}
The biased Bouchaud's trap model on $\Z$ is the special case where $(w_x)_{x \in \Z}$ is an i.\,i.\,d.\  sequence, and where 
$\tau_{r,x,k} = r \e_{x,k}$, and $(\e_{x,k})_{x \in \Z, \, k \in \N}$ is a family of i.\,i.\,d.\ exponentially distributed random variables with mean $1$;
see \cite{BenArous+Fribergh:2016} for further details on Bouchaud's trap model.
Note that the exponential distribution is necessary (and sufficient) to make $(X_t)_{t \geq 0}$ a continuous time Markov process,
but this property is irrelevant for our purposes.

There is a strong law of large numbers for our model.
To formulate it, we introduce the notation $\pesc = \pesc(\lambda)$ for the escape probability of $(Y_n)_{n \in \N_0}$
from the origin, i.e.,
\begin{align*}
\pesc \defeq \Prob(Y_n \not = 0 \text{ for all } n \in \N) = 2p_\lambda - 1
= \frac{2{\rm e}^\lambda}{{\rm e}^\lambda + {\rm e}^{-\lambda}} -1
= \frac{{\rm e}^{\lambda}-{\rm e}^{-\lambda}}{{\rm e}^\lambda + {\rm e}^{-\lambda}}.
\end{align*}


\begin{theorem}	\label{Thm:speed explicit}
Let $(w_x)_{x \in \Z}$ be ergodic under $\Prob$, and let $\E$ be the expectation with respect to $\Prob$. Then
$\vel := \lim_{t \to \infty} \frac{X_t}{t}$ exists almost surely, and we have 
\begin{equation}	\label{eq:lim X_t/t}
\vel = \frac{\pesc(\lambda)}{\E[w_0]} 
\quad	\text{almost surely.}
\end{equation}
\end{theorem}

Formula \eqref{eq:lim X_t/t} has the simple interpretation
that $(Y_n)_{n \in \N_0}$ travels, by the strong law of large numbers, with linear speed $2p_\lambda -1 = \pesc(\lambda)$.
The average duration of each visit of the walk $(X_t)_{t \geq 0}$ to a site $x$ is equal to $\E[w_x] = \E[w_0]$. 
So on average, the walk makes a step every $1/\E[w_0]$ units of time
resulting in a speed of $\pesc(\lambda)/\E[w_0]$.

\subsubsection*{{\bf Example 1:} One-dimensional walk with vertical traps.}

\noindent
Consider the graph obtained by attaching below each point $x$ of $\Z$ (pictured as lying on a horizontal line, called the backbone) 
a finite one-dimensional `branch' (branches are pictured as lying on vertical lines, called traps, see Figure \ref{Fig:Example 1} below)
of random length $L_x \in \N_0$.
Assume that the $L_x$, $x \in \Z$ are i.\,i.\,d. We run a biased 
discrete-time random walk on this graph, with the following transition probabilities: when $x$ is on the backbone and 
the length of the trap is positive, it goes to the right with probability $q p_\lambda$, to the left with probability 
$q (1 - p_\lambda)$, and into the trap with probability $1-q$ for some fixed $q \in (0,1)$. When in a trap, it goes down the trap 
with probability $p_\lambda$ and up the trap with probability $1 - p_\lambda$, unless it is at the very bottom of the trap, 
at which point is goes up with probability one.  \\[1mm]

\begin{figure}[H]
\begin{center}
\begin{tikzpicture} [level distance=8mm] 
 \tikzstyle{every node}=[fill=black, circle,inner sep=1.25pt]
 
\node  at (2,0) {} [grow = right]
child[grow=left, dotted] {node [fill=none] {}}
	child[grow=down, ] {node {}
	child[grow=down, ] {node {}}}
    child [grow=right, draw] {[fill] circle (1.5pt)
   	child[grow=right] {node (i) {}
	 child[grow=down] {node (h) {}}
		child [grow=right] {node (p) {}
		child[grow=down] {node (g) {}
			child[grow=down] {node (j) {}}}
			child [grow=right] { node (l) {}
				child[grow=down] {node {}
	 	child[grow=down] {node {}
		child[grow=down] {node {}}}}
				child [grow=right] {node [draw, fill=black]  (d){}
					child[grow=down] {node {}
					child[grow=down] {node {}}}
					child [grow=right] { node [draw, fill=gray]  (a){}
						child[grow=down, fill=black] {node (b) {}
						child[grow=down] {node {}
						child[grow=down] {node {}
						child[grow=down] {node {}}}}}
						child [grow=right]  {node [draw, fill=black]  (c){}
						child[grow=down] {node {}}
							child [grow=right] {node {}
							child [grow=right] {node {}
								child[grow=down] {node {}
								child[grow=down] {node {}
								child[grow=down] {node {}}}}
								child [grow=right, dotted] {node [fill=none] {}}
								}
							}}}}}}}};
							
\node[fill=none] at (6.1,-0.4)[label=-40: \; \textcolor{gray}{$1-q$} \; ]{};
\node[fill=none] at (6.5, 0.75)[label=-40: \; \textcolor{black}{$x$} \; ]{};
\node[fill=none] at (4.9, 1.2)[label=-40: \; \textcolor{gray}{$q(1-p_{\lambda})$} \; ]{};
\node[fill=none] at (7.2, 0.85)[label=-40: \; \textcolor{gray}{$qp_{\lambda}$} \; ]{};

			
		\draw[->, black,thick,gray] (a) to (b);
		\draw[->, black,thick,gray] (a) to (c);
		\draw[->, black,thick,gray] (a) to (d);	
						
\end{tikzpicture}
\caption{Part of the graph, where the walk is on the backbone at the highlighted vertex $x$.}	\label{Fig:Example 1}
\end{center}
\end{figure}
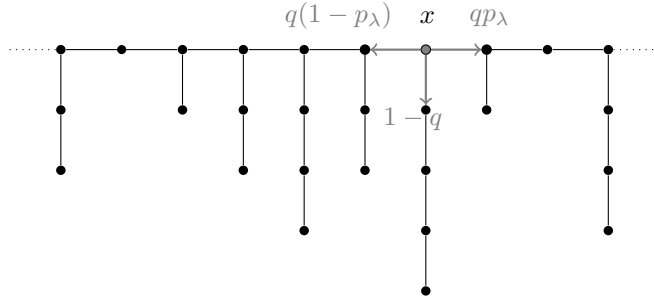

\noindent
This model can be embedded into our model by simply only considering the horizontal coordinate of the walk. 
Then, the walk spends a random amount of time on or below the vertex $x$ of the backbone before moving onto a neighboring vertex, 
the times at different vertices are independent, and have possibly different distributions, and the times spent at consecutive 
visits at the same vertex are also independent, but have the same distribution at every visit. 
All that remains to be done is to calculate $w_x$ and $(\tau_{r,x,k})_{r \geq 0}$.
To this end, assume that we have a trap of length 
$\ell \geq 1$ at some vertex. Let $t_\ell$ be the random time needed by the walk to exit the trap conditional on making one step into it. 
Although the distribution of $t_\ell$ is complicated, its expectation can be computed by electrical network methods, and the result is 
\begin{equation*}
\E[t_\ell] = 2 \frac{{\rm e}^{2 \lambda \ell} - 1}{{\rm e}^{2 \lambda} - 1},
\end{equation*}
see \cite[Eq.\ (1.5)]{Gantert+Klenke:2022}, and note that in their notation, $\beta = {\rm e}^{2\lambda}$.
Since the walk can enter a trap multiple times 
without moving horizontally, the distribution of the total time spent at a trap of depth $\ell$ is equal to the distribution of 
$S_\ell = 1 + \sum_{k=1}^N t_{\ell,k}$, where $N$ is geometrically distributed with success probability $q$
(i.\,e., $\Prob(N=n) = (1-q)^n q$, $n \in \N_0$)
and where the $t_{\ell,k}$, $k \in \N$ are independent copies of $t_\ell$.
It follows that the expected time spent at a site with a trap of length 
$\ell$ is equal to 
\begin{equation*}
\bar S_\ell \defeq \E[S_\ell]
= 1 + \frac{1-q}{q} \E[t_\ell]
= 1 +  \frac{2(1-q)}{q} \frac{{\rm e}^{2 \lambda \ell} - 1}{{\rm e}^{2 \lambda} - 1}.
\end{equation*}
Therefore, with $S_\ell = 1$ if $\ell=0$, we must choose $w_x$ as i.\,i.\,d.\  random variables with 
\begin{equation}
\label{eq:def w_x}
\Prob(w_x = \bar S_\ell) = \Prob(L = \ell),
\end{equation}
and $\tau_{\bar S_\ell,x,k} = S_{\ell,x,k}$, where $S_{\ell, x, k}$ are independent copies of $S_{\ell}$. 
We extend $\tau_{r,x,k}$ to $r \notin \{\bar S_\ell: \ell \in \N_0\}$ piecewise linear to make $(\tau_{r,x,k})_{r \geq 0}$ (right-) continuous
and to ensure $\E[\tau_{r,x,k}]=r$ for all $r \geq 0$.
(Actually, we can choose the $t_{\ell,k}$ non-decreasing in $\ell$ and hence $(\tau_{r,x,k})_{r \geq 0}$ non-decreasing in $r$,
but this is not required here.)
We then get 
\begin{equation*}
\vel = \frac{\pesc(\lambda)}{\E[w_0]} = \frac{\pesc(\lambda)}{\sum_{\ell=0}^\infty \bar S_\ell \Prob(L=\ell)} =
 \frac{{\rm e}^\lambda - {\rm e}^{-\lambda}}{{\rm e}^\lambda + {\rm e}^{-\lambda}} 
\frac{1}{1 + \frac{2(1-q)}{q} \frac{\E[{\rm e}^{2 \lambda L}] - 1}{{\rm e}^{2\lambda}-1}}.
\end{equation*}
We observe that the distribution of the trap length $L$ plays a crucial role.
If $L$ has all exponential moments, then 
the speed is strictly positive for all $\lambda > 0$, while in the case where $L$ has no exponential moments,
the speed is always 
equal to zero. The phase transition from positive to zero speed occurs if $L$ has some but not all exponential moments.
This is the case,
for instance, if $L$ is geometrically distributed, which seems to be the case
(at least approximately) in most of the important models -- see Example 3 below and 
the discussion following it.
We set $\Prob(L = \ell) = (1 - {\rm e}^{-\alpha}) {\rm e}^{-\alpha \ell}$, $\ell \in \N_0$.
In \cite{Gantert+Klenke:2022}, 
precise tail estimates for the random variable $\sum_{k=0}^L t_{\ell,k}$ are obtained for this choice of $L$, 
but for our purposes, simple calculations suffice. 
We obtain 
\begin{equation*}
\E[{\rm e}^{2 \lambda L}] -1
= \frac{1 - {\rm e}^{-\alpha}}{1 - {\rm e}^{2 \lambda - \alpha}} -1
= \frac{{\rm e}^{2\lambda}-1}{{\rm e}^\alpha - {\rm e}^{2\lambda}} \qquad \text{if } \lambda < \alpha/2,
\end{equation*}
and equal to infinity otherwise.  We conclude that for fixed $\alpha >0$ and $0 < \lambda < \alpha/2$, 
\begin{equation*}
\vel(\lambda) =  \frac{{\rm e}^\lambda - {\rm e}^{-\lambda}}{{\rm e}^\lambda + {\rm e}^{-\lambda}} 
\frac{{\rm e}^\alpha - {\rm e}^{2 \lambda}}{{\rm e}^\alpha - {\rm e}^{2 \lambda} + \frac{2(1-q)}{q}},
\end{equation*}
whereas $\vel = 0$ if $\lambda \geq \alpha/2$.

\subsubsection*{{\bf Example 2:} Bouchaud's trap model with drift-dependent holding times.}

We can reduce the previous example to Bouchaud's trap model by replacing the precise distribution of 
$S_\ell = 1 + \sum_{k=0}^L t_{\ell,k}$ with 
a general heavy-tailed distribution for $w_0$, and by setting $\tau_{r,x,k} = r \e_{x,k}$ with $\e_{x,k}$ exponentially distributed with mean $1$.
The main features are the same: inspired by the known tail behaviour of $S_\ell$ as a function of the bias
\cite{Gantert+Klenke:2022,Luebbers+Meiners:2019}, we set 
\begin{equation}	\label{eq:tau}
\Prob(w_0 \geq t) = t^{-\alpha}	\quad	\text{for } t \geq 1,
\end{equation}
for some $\alpha > 0$.
We obtain
\begin{equation*}
\vel(\alpha,\lambda) = \frac{e^\lambda-e^{-\lambda}}{e^\lambda+e^{-\lambda}} \frac{(\alpha-1)^+}{\alpha}.
\end{equation*}
We choose $\alpha = \lambdacrit/\lambda$ where $\lambdacrit>0$ is a parameter,
and write $\vel(\lambda)$ for $\vel(\alpha,\lambda)$ if \eqref{eq:tau} holds with $\alpha=\lambdacrit/\lambda$.
In particular, if $\alpha = \lambdacrit/\lambda$,
\begin{equation*}
\vel(\lambda) = \frac{e^\lambda-e^{-\lambda}}{e^\lambda+e^{-\lambda}} \frac{(\lambdacrit-\lambda)^+}{\lambdacrit}.
\end{equation*}

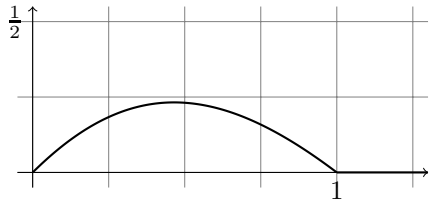
\begin{figure}[h]
\begin{center}
	\begin{tikzpicture}[scale=1]
		\draw [help lines] (-0.2,-0.2) grid (5.2,2.2); 
		\draw (0,2)	node[left] {$\tfrac12$};
		\draw (4,0)	node[below] {$1$};
		\draw [->] (-0.2,0) -- (5.2,0); 
		\draw [->] (0,-0.2) -- (0,2.2);
		\draw [thick, domain=0:4, samples=100] plot (\x, {4*(exp(\x/4)-exp(-\x/4))/(exp(\x/4)+exp(-\x/4))*(1-\x/4))});
		\draw [thick, domain=4:5.2, samples=100] plot (\x, {0});
	\end{tikzpicture}
	\caption{The function $\lambda \mapsto \vel(\lambda)$ with $\lambdacrit \defeq 1$.}
\end{center}
\end{figure}

\noindent
\subsubsection*{{\bf Example 3:} One-dimensional percolation with horizontal traps.}
As in Example 1, we consider a random graph with a backbone that is just $\Z$, pictured as lying on a horizontal line, 
and we attach traps at each point of the backbone. The difference now is while the first vertex of a trap is still 
below the vertex of $\Z$ at which the trap is attached, all further vertices of the trap are to the right of that first 
vertex, below some other vertices of the backbone. We demand that no other traps of positive length can occur for 
$\ell-1$ steps to the right of a trap of length $\ell$, meaning that we can represent the random graph as a subgraph of 
$\Z \times \{0,1\}$ (with nearest-neighbor edges).

\begin{figure}[H]
\begin{center}
\begin{tikzpicture} [level distance=8mm] 
 \tikzstyle{every node}=[fill=black, circle,inner sep=1.25pt]

\node  at (2,0) {} [grow = right]
child[grow=left, dotted] {node [fill=none] {}}
	child[grow=down] {node (j) {}
	child [grow = right] {node (p) {}
	child [grow = right] {node (r) {}
	child [grow = right] {node (s) {}}
	}}}
	child[grow=right] {node (a) {}
	child[grow=right] {node (b) {}
	child[grow=right] {node (c) {}
	child[grow=right] {node (d) {}
	child[grow=down] {node (j) {}
	child [grow = right] {node (t) {}}}
	child[grow=right] {node (e) {}
	child[grow=right] {node (f) {}
	child [grow = down] {node (t) {}}
	child[grow=right] {node (g) {}
	child[grow=down] {node (j) {}
	child [grow = right] {node (p) {}
	child [grow = right] {node (r) {}}}}
	child[grow=right] {node (h) {}
	child[grow=right] {node (i) {}
	child[grow=right, dotted] {node [fill=none] {}}
	}}}}}}}}};

												
\end{tikzpicture}
\caption{Part of the graph.}	\label{Fig:Example 3}
\end{center}
\end{figure}
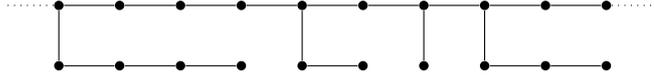

\noindent
In fact, this is the model studied in 
\cite{A-F+H:2009,Gantert+al:2018} with traps only or single edges to the right. 
If we assume that edges not in the backbone appear independently with probability ${\rm e}^{-\alpha}$, then conditionally on having 
only a backbone and traps, the trap length $L$ is geometric with success probability $1-{\rm e}^{-\alpha}$. 
Inside a trap, we assume the same dynamics as in Example 1, 
i.e.\ a bias towards the end of the trap
(which is now to the right, so the bias has the same direction and strength as in the backbone). 
The only point where greater generality than given in Example 1 might be desirable is the probability of the final jump out of the 
trap, which is in the vertical direction, and one might want to assign it a different probability than 
${\rm e}^{-\lambda}/({\rm e}^{\lambda}+{\rm e}^{-\lambda})$.
We do not do this for the sake of not further complicating affairs. 

The model described above can again be mapped exactly into the context of Theorem \ref{Thm:speed explicit}. To do this, consider a 
Markov chain $(Z_x)_{x \in \Z}$ with values in $\N_0$ and with transition probabilities 
\begin{equation*}
p(i,j) = \begin{cases} 
\Prob(L=j) 	& \text{if } i \in \{0,1\}, \\
1 			& \text{if } i>1,\;j=i-1,\\
0 			& \text{otherwise.}
\end{cases}
\end{equation*}
The interpretation of this chain is that when $Z_x \geq 1$, then there is a trap below $x$ (possibly beginning to the left of $x$)
with $Z_x$ remaining vertices (including the vertex below $x$).
Traps of length $1$ do not obstruct further traps, but traps of length $\ell \geq 2$ prevent further traps for $\ell-1$ steps.
Therefore, if we define 
$f:\N_0^2 \to \N_0$ with 
\begin{equation*}
f(i,j) = \begin{cases} 
j & \text{if } j \geq i,\\
0 & \text{if } j < i,
\end{cases}
\end{equation*}
then $\Lambda_x = f(Z_{x-1},Z_x)$ models the length of the trap rooted at site $x$. 

The Markov chain $(Z_x)_{x \in \Z}$ has a unique invariant measure
if and only if $\E[L]<\infty$: from the transition probabilities, we conclude that 
the weight 
function $\pi$ of an invariant measure must fulfil the equations
$\pi(0) \Prob(L=0)  + \pi(1) \Prob(L=0) = \pi(0)$, and 
\begin{equation*}
\pi(0) \Prob(L=j) + \pi(1) \Prob(L=j) + \pi(j+1) = \pi(j) \qquad \text{for } j \in \N.
\end{equation*}
The unique probability weight function solving these equations is given by 
\begin{equation*}
\pi(0) = \frac{\Prob(L=0)}{\Prob(L=0) + \E[L]}, \quad \pi(j) = \frac{\Prob(L \geq j)}{\Prob(L=0) + \E[L]} \quad	\text{ for } j \in \N.
\end{equation*}
The stationary Markov chain $(Z_x)_{x \in \Z}$ is thus ergodic (see e.g.\ Theorem 5.2.6 of \cite{Douc+al:2018}),
and by \cite[Lemma 5.6(c)]{A-F+H:2009} and the representation $\Lambda_x = f(Z_{x-1},Z_x)$, also 
the process $(\Lambda_x)_{x \in \Z}$ is ergodic. Defining $\bar S_\ell$ as in Example 1 and setting  
$w_x = \bar S_{\Lambda_x}$, we see (e.g.\ again by \cite[Lemma 5.6(c)]{A-F+H:2009}) that also 
$(w_x)_{x \in \Z}$ is ergodic. 
After defining $\tau_{\bar S_\ell, x,k} = S_{\ell,x,k}$ in the same way as in Example 1, we are back in the setting of 
Theorem \ref{Thm:speed explicit}. It remains to compute $\E[w_0]$. We have 
\begin{equation*}
\Prob(\Lambda_0 = j) = \sum_{k=0}^\infty \Prob(f(Z_{-1},Z_0) = j \, | \, Z_{-1} = k) \pi(k) = \Prob(L = j) (\pi(0) + \pi(1)) + 
\delta_{j,0} (1 - \pi(0) - \pi(1)).
\end{equation*}
Since $\pi(0) + \pi(1) = \frac{1}{\Prob(L=0) + \E[L]}$ and $\bar S_0 = 1$, we obtain 
\begin{equation*}
\E[w_0] = \sum_{j = 0}^\infty \bar S_j \Prob(\Lambda_0 = j) = \frac{\E[\bar S_L]}{\Prob(L=0) + \E[L]} + 
\bar S_0 (1 - \tfrac{1}{\Prob(L=0) + \E[L]}) = 1 + \frac{\E[\bar S_L]-1}{\Prob(L=0) + \E[L]}. 
\end{equation*}
Assuming again that $L$ is geometric with success probability $1-{\rm e}^{-\alpha}$, and $\lambda < \alpha/2$, we know from Example 1 that 
\begin{equation*}
\E[\bar S_L] = 1 + \frac{2(1-q)}{q} \frac{1}{{\rm e}^{\alpha} - {\rm e}^{2\lambda}},
\end{equation*}
and we calculate $\Prob(L=0) + \E[L] = 1 + \frac{{\rm e}^{-2\alpha}}{1 - {\rm e}^{-\alpha}}$.
This gives 
\begin{equation*}
\E[w_0] = 1 + \frac{2(1-q)}{q} \frac{1}{({\rm e}^{\alpha} - {\rm e}^{2\lambda}) (1 + \frac{{\rm e}^{-2\alpha}}{1 - {\rm e}^{-\alpha}})},
\end{equation*}
and finally we obtain an  explicit formula for
\begin{equation*}
\vel(\lambda) =   \frac{{\rm e}^\lambda - {\rm e}^{-\lambda}}{{\rm e}^\lambda + {\rm e}^{-\lambda}}  \frac{1}{\E[w_0]}.
\end{equation*}

\begin{remark}
We have seen in Examples 1 and 3 that the tail behaviour of the 
geometric distribution of the trap length is crucial for the slowdown to zero of the speed for a finite bias. Proposition 1.1
of \cite{Fribergh+Hammond:2014} shows that this tail is also present in the case of the biased walk on a percolation cluster, 
although in this case a direct mapping to our reduced model is probably not possible. For the full model of one-dimensional percolation 
on the ladder graph \cite{A-F+H:2009,Gantert+al:2018}, besides the traps that we treat in Example 3 there are areas that are not traps, 
i.\,e.\ where it is possible to go to the next trap entrance to the right without making a step to the left. In such areas, the speed of the biased 
random walk should be non-decreasing as a function of the bias. For an explicit formula for the speed in the ladder graph model, one would need 
to find the probabilities of all possible non-trap areas along with the average time it takes to travel through them, then combine this 
with the considerations of Example 3. Using some of the considerations in \cite{A-F+H:2009}, this seems possible, but will be messy, 
so we do not do it. We anyway expect that the slowdown behavior will be of the same type, since it is dominated by the trap regions.
\end{remark}

\section{Proof of Theorem \ref{Thm:speed explicit}}

While Theorem \ref{Thm:speed explicit} is highly plausible,
a rigorous proof is required.
We use a coupling-from-the-past argument adapted from \cite{A-F+H:2009},
in which a formula,  not tractable for explicit calculations, for the speed of biased random walk
in a one-dimensional percolation environment is derived.

On a suitable probability space, consider independent families of random variables 
$(w_x)_{x \in \Z}$,  $((\tau_{r,x,k})_{r \geq 0})_{x \in \Z, k \in \N_0}$ and $(U_x(k))_{x \in \Z,\,k \in \N}$ where $U_x(k)$ is uniform on $(0,1)$, 
$(w_x)_{x \in \Z}$ is ergodic, the processes $(\tau_{r,x,k})_{r \geq 0}$, $x \in \Z, k \in \N_0$ are i.\,i.\,d.\ $r,x,k$ and $\E[\tau_{r,x,k}] = r$ for all $r \geq 0$.
We note already here that 
by \cite[Lemma 5.6(b)]{A-F+H:2009}, the family 
\begin{equation*}
(\boldsymbol V_x)_{x \in \Z} \defeq ((U_x(k))_{k \in \N}, w_x, ((\tau_{r,x,k})_{r \geq 0})_{k \in \N})_{x \in \Z}
\end{equation*}
is ergodic. Let us write $E \defeq (0,1)^\N \times \R^+ \times D^\N$ for the target space of the $\boldsymbol V_{\!x}$.
By \cite[Lemma 5.6(b)]{A-F+H:2009}, 
for any measurable function $f: E^{\N} \times E^{\N} \to \R$, the sequence 
\begin{equation*}
(f(\boldsymbol V_{\!x}, \boldsymbol V_{\!x+1}, \boldsymbol V_{\!x+2}, \ldots; \boldsymbol V_{\!x-1}, \boldsymbol V_{\!x-2}, \ldots))_{x \in \Z}
\end{equation*} 
is ergodic. We will make extensive use of this fact.   

First, for any $x \in \Z$, we may define $(Y_n^x)_{n \in \N_0}$
as the discrete-time random walk started at $x$ that, when hitting a state $y \in \Z$
for the $k^{\mathrm{th}}$ time, makes the next step from $y$ to $y-1$ if
\begin{equation*}
U_{y}(k) \leq \frac{e^{-\lambda}}{e^\lambda + e^{-\lambda}},
\end{equation*}
and steps to $y+1$, otherwise.
A point $x \in \Z$ is called a \emph{regeneration point} if $Y_n^x > x$ for all $n \in \N$.
Write $I_x \defeq \1_{\{Y_n^x > x \text{ for all } n \in \N\}}$ for the indicator of the event that $x$ is a regeneration point.

\begin{lemma}	\label{Lem:I_x}
The sequence $(I_x)_{x \in \Z}$ is ergodic and
\begin{align}	\label{eq:Birkhoff I_x}
\lim_{n \to \infty} \frac1n \sum_{x=0}^{n-1} I_x = r_\lambda	\quad	\text{almost surely}
\end{align}
for $r_\lambda \defeq \Prob(I_0) > 0$.
\end{lemma}

We write $\bU_{\!x}$ for $(U_x(k))_{k \in \N}$.

\begin{proof}
Notice that  $I_x = f(\bU_{\!x},\bU_{\!x+1},\ldots;\bU_{\!x-1},\bU_{\!x-2},\ldots)$ for some Borel measurable function
$f: ((0,1)^\N)^\Z \to [0,1]$ (actually, $f$ only depends on the variables $\bU_{\!x},\bU_{\!x+1},\ldots$). 
Then, since the family $(\bU_{\!y})_{y \in \Z}$ is ergodic
as a family of i.i.d.\ random variables on $(0,1)^\N$,
we conclude that $(I_x)_{x \in \Z}$ is also ergodic by \cite[Lemma 5.6(c)]{A-F+H:2009}.

Since $\lambda>0$, each of the walks $(Y_n^x)_{n \in \N_0}$ is transient to the right
and, thus, $r_\lambda = \Prob(I_0) > 0$.
Now \eqref{eq:Birkhoff I_x} follows from Birkhoff's ergodic theorem.
\end{proof}

By \eqref{eq:Birkhoff I_x}, $\Prob(I_x = 1 \text{ for infinitely many } x \in \N_0) = 1$
and hence, by shift invariance and with some effort, we conclude $\Prob(I_x = 1 \text{ for infinitely many } x < 0) = 1$.
We enumerate the random points $x \in \Z$ with $I_x=1$ from left to right
with $\ldots,R_{-1}, R_0, R_1, R_2,\ldots$ such that $R_k < R_{k+1}$ for all $k \in \Z$
and $R_{-1} < 0 \leq R_0 < R_1$.
Now notice that for any $x,y \in \Z$, the trajectories of $(Y_n^x)_{n \in \N_0}$ and
$(Y_n^y)_{n \in \N_0}$ coincide once they hit the first $R_k \geq x \vee y$.
For $k \in \Z$, define $\rho_k \defeq \inf\{n \in \N_0: Y_n^{R_k} = 0\}$ to be the first time
at which the random walk started at $R_k$ hits $0$. By the transience to the right,
$\rho_k < \infty$ for all $k<0$.
Also, at time $\rho_{k-1}-\rho_{k}$ the walk started at $R_{k-1}$
hits $R_{k}$ and its trajectory then coincides with the walk started at $R_{k}$ for $k<0$,
i.e.,
\begin{equation}	\label{eq:no conflict}
(Y^{R_{k-1}}_{\rho_{k-1}-\rho_{k}}, Y^{R_{k-1}}_{\rho_{k-1}-\rho_{k}+1},\ldots)
= (Y^{R_{k}}_0, Y^{R_{k}}_1,\ldots).
\end{equation}
Consequently, we may set
\begin{align*}
(Y_{-\rho_k}^{-\infty},Y_{-\rho_k+1}^{-\infty},\ldots) \defeq (Y_0^{R_{k}},Y_1^{R_{k}},\ldots)
\end{align*}
for $k<0$. As $\rho_k \to \infty$ almost surely as $k \to -\infty$,
this defines $Y_n^{-\infty}$ for all $n \in \Z$ almost surely,
as we may choose $k < 0$ (randomly) such that $-\rho_k < n$.
Notice that when $-\rho_k < n$, then also $-\rho_{k-1} < n$
and $Y_n^{-\infty}$ is defined twice (actually infinitely often).
Then  \eqref{eq:no conflict} guarantees that the definitions coincide, namely,
\begin{equation*}
Y^{R_{k-1}}_{\rho_{k-1}+n} = Y^{R_{k-1}}_{\rho_{k-1}-\rho_k+ \rho_k + n} = Y^{R_k}_{\rho_k+n}.
\end{equation*}
We may assume without loss of generality that $(Y_n)_{n \in \N_0} = (Y_n^0)_{n \in \N_0}$.
We then define $(X_t)_{t \geq 0}$ as the continuous-time process (with right-continuous paths)
starting at the origin at time $0$
that follows the trajectory of $(Y_n)_{n \in \N_0}$ but stays at $x \in \Z$ upon its $k^{\mathrm{th}}$ visit to this point
for time $\tau_{w_x,x,k}$. (This process has the same law as the process defined in Section \ref{sec:results}.)
Similarly, we define $(X_t^{-\infty})_{t \geq 0}$ as the continuous-time process (with right-continuous paths)
that follows the trajectory of $(Y_n^{-\infty})_{n \in \N_0}$,
hits the origin for the first time at time $0$ and stays at $x \in \Z$ upon its $k^{\mathrm{th}}$ visit to this point
for time $\tau_{w_x,x,k}$.

For $x \in \Z$, we define
\begin{equation*}
N_x \defeq \sum_{n \in \Z} \1_{\{Y_n^{-\infty} = x\}}
\quad	\text{and}	\quad
Z_x \defeq \sum_{k=1}^{N_x} \tau_{w_x,x,k} 
\end{equation*}
to be the number of visits of the walk $(Y_n^{-\infty})_{n \in \Z}$ to $x$ and the time the continuous-time random walk
$(X_t^{-\infty})_{t \in \R}$ spends at $x$, respectively.

\begin{lemma}	\label{Lem:N_x+Z_x}
The families $(N_x)_{x \in \Z}$ and $(Z_x)_{x \in \Z}$ are ergodic.
In particular, almost surely as $n \to \infty$,
\begin{align}	\label{eq:Birkhoff N_x+Z_x}
\frac1n \sum_{x=0}^{n-1} N_x \to \E[N_0] = \frac{e^\lambda+e^{-\lambda}}{e^\lambda-e^{-\lambda}}
\quad	\text{and}	\quad
\frac1n \sum_{x=0}^{n-1} Z_x \to \E[Z_0] = \frac{e^\lambda+e^{-\lambda}}{e^\lambda-e^{-\lambda}} \E[w_0].
\end{align}
\end{lemma}
\begin{proof}
The limiting relations in \eqref{eq:Birkhoff N_x+Z_x} follow from Birkhoff's
ergodic theorem once we have shown the ergodicity of $(N_x)_{x \in \Z}$ and $(Z_x)_{x \in \Z}$
and have calculated $\E[N_0]$ and $\E[Z_0]$.
It is worth mentioning that, using truncation techniques, Birkhoff's ergodic theorem also applies in the case $\E[w_0]=\infty$.
Regarding $\E[N_0]$, notice that the number of returns to the origin of $(Y_n^{-\infty})_{n \in \N_0}$
is geometric with success probability
being the escape probability $\pesc$. Consequently,
\begin{align*}
\E[N_0] = \sum_{n=1}^\infty \Prob(N_0 \geq n) = \sum_{n=1}^\infty (1-\pesc)^{n-1}
= \frac1\pesc
= \frac{e^\lambda+e^{-\lambda}}{e^\lambda-e^{-\lambda}}.
\end{align*}
Further, since $N_x$ and $(\tau_{w_x,x,k})_{k \in \N}$ are independent,
by Wald's identity,
\begin{align*}
\E[Z_0] = \E[N_0] \E[\tau_{w_0,0,1}] = \E[N_0] \E[w_0].
\end{align*}
It remains to prove the ergodicity of $(N_x)_{x \in \Z}$ and $(Z_x)_{x \in \Z}$.
To this end, first notice that $N_0 = g(\bU_{\!0},\bU_{\!1},\ldots;\bU_{\!-1},\bU_{\!-2},\ldots)$ for some Borel measurable function
$g: ((0,1)^\N)^\Z \to \R$.
Then, since $N_x = g(\bU_{\!x},\bU_{\!x+1},\ldots;\bU_{\!x-1},\bU_{\!x-2},\ldots)$
and since the family $(\bU_{\!y})_{y \in \Z}$ is ergodic,
so is $(N_x)_{x \in \Z}$ again by Lemma 5.6(c) of \cite{A-F+H:2009}.
Regarding the ergodicity of $(Z_x)_{x \in \Z}$, define 
\begin{equation*}
h: \N \times \R \times D^\N \to \R, \quad (n,r,((x_{s,k})_{s \geq 0})_{k \in \N}) \mapsto \sum_{j=1}^n x_{r,j},
\end{equation*}
then $h$ is product-measurable. This is true since the projections are measurable with respect to the Borel-$\sigma$-field $\mathcal{D}$
on the Skorohod space $D$ equipped with the Skorohod topology, see \cite[Theorem 16.6]{Billingsley:1999}.
By the right-continuity of the elements of the Skorohod space,
also $\R^+ \times D \ni (t,x) \mapsto x(t) \in \R$ is $(\B(\R^+) \otimes \mathcal{D}$)-$\B(\R)$-measurable and using this,
the product-measurability follows from standard arguments.
Since we have
\begin{equation*}
Z_x = h(g(\bU_{\!x},\bU_{\!x+1},\ldots;\bU_{\!x-1},\bU_{\!x-2},\ldots), w_x, ((\tau_{s,x,k})_{s \geq 0})_{k \in \N}),
\end{equation*}
the ergodicity of $(Z_x)_{x \in \Z}$ follows as above from the ergodicity of $(\bU_x, w_x, ((\tau_{r,x,k})_{r \geq 0})_{k \in \N})_{x \in \Z}$
and again Lemma 5.6(c) of \cite{A-F+H:2009}.
\end{proof}

\begin{proof}[Proof of Theorem \ref{Thm:speed explicit}]
For $n \in \N_0$, let $\varrho_n \defeq \inf\{t \geq 0: X_t = R_n\}$ and,
similarly, $\varrho_n^{-\infty} \defeq \inf\{t \geq 0: X_t^{-\infty} = R_n\}$.
Once the random walks hit $R_0$, the first regeneration point on the nonnegative axis,
the trajectories of $(X_t)_{t \geq 0}$ and $(X_t^{-\infty})_{t \geq 0}$ coincide, i.e.,
\begin{equation*}
(X_{\varrho_0+t})_{t \geq 0} = (X^{-\infty}_{\varrho^{-\infty}_0+t})_{t \geq 0}.
\end{equation*}
(Notice that this is not the case with $\varrho_0$ and $\varrho^{-\infty}_0$ replaced by $0$.)
Let $\Delta \defeq \varrho_0 - \varrho_0^{-\infty}$.
Then, for $t > \varrho_0 \vee \varrho^{-\infty}_0$,
\begin{align*}
\frac{X_t}{t} = \frac{X_{\varrho_0+t-\varrho_0}}{t} = \frac{X_{\varrho_0^{-\infty}+t-\varrho_0}^{-\infty}}{t}
= \frac{X_{t-\Delta}^{-\infty}}{t-\Delta} \frac{t-\Delta}{t}.
\end{align*}
Therefore, it suffices to prove that
\begin{equation}	\label{eq:lim X_t^-infty/t}
\lim_{t \to \infty} \frac{X_t^{-\infty}}{t} = \frac{1}{\E[N_0] \E[w_0]}	\quad	\text{almost surely}.
\end{equation}
Since the time the continuous-time random walk $(X_t^{-\infty})_{t \geq 0}$ spends on the negative half-axis is finite by transience
and since $R_n \to \infty$ almost surely as $n \to \infty$,
we conclude that
\begin{align}	\label{eq:limrho_n/R_n}
\lim_{n \to \infty} \frac{\varrho_n^{-\infty}}{R_n}
= \lim_{n \to \infty} \frac{1}{R_n} \sum_{x=0}^{R_n-1} Z_x
= \E[Z_0] = \E[N_0] \E[w_0]	\quad	\text{almost surely.}
\end{align}
This may be rewritten as 
\begin{align}	\label{eq:limR_n/rho_n}
\lim_{n \to \infty} \frac{1}{\varrho_n^{-\infty}} X^{-\infty}_{\varrho_n^{-\infty}} = \lim_{n \to \infty} \frac{R_n}{\varrho_n^{-\infty}}
= \frac{1}{\E[N_0] \E[w_0]}	\quad	\text{almost surely}
\end{align}
and gives the desired convergence \eqref{eq:lim X_t^-infty/t}, but only as $t \to \infty$
along the regeneration times $\varrho_1^{-\infty},\varrho_2^{-\infty},\ldots$.
This convergence along a subsequence can be lifted to \eqref{eq:lim X_t^-infty/t}
by a standard sandwich argument. For the reader's convenience,
we give this argument.
For $t > 0$, let $t_+ \defeq \inf\{\varrho_n^{-\infty}: n \in \N_0, R_n> X_t\}$ and $t_- \defeq \max\{\varrho_n^{-\infty}: n \in \N_0, R_n \leq X_t\}$
where the maximum of the empty set is defined to be $0$.
From \eqref{eq:Birkhoff I_x}, we infer
\begin{align}	\label{eq:limR_n/n}
\lim_{n \to \infty} \frac{R_n}n = \lim_{n \to \infty} \bigg(\frac1{R_n} \sum_{x=0}^{R_n-1} I_x \bigg)^{-1} = \frac1{r_\lambda}
\quad	\text{almost surely}.
\end{align}
in particular, $R_{n+1}/R_n \to 1$ almost surely.

We distinguish two cases, namely, $\E[w_0]<\infty$ and $\E[w_0]=\infty$.
First suppose that $\E[w_0]=\infty$.
Then, for $t > \varrho_0^{-\infty}$, we have
\begin{align}	\label{eq:sandwich}
\frac{X_t^{-\infty}}t  = \frac{X_{t}^{-\infty}}{X_{t_-}^{-\infty}} \frac{X_{t_-}^{-\infty}}{t_-} \frac{t_-}{t}.
\end{align}
The first factor tends to $1$ as $t \to \infty$ almost surely since $R_{n+1}/R_n \to 1$ almost surely.
Further, $t_-/t$ is bounded by one whereas $X_{t_-}^{-\infty}/t_- \to 0$ almost surely by \eqref{eq:limR_n/rho_n},
so \eqref{eq:lim X_t^-infty/t} holds.

Now suppose $\E[w_0]<\infty$.
Then, for $t > \varrho_0^{-\infty}$, we have
\begin{align}	\label{eq:sandwich}
\frac{X_{t_-}^{-\infty}}{t_-} \frac{t_-}{t_+}
= \frac{X_{t_-}^{-\infty}}{t_+} \leq \frac{X_t^{-\infty}}t  \leq \frac{X_{t_+}^{-\infty}}{t_-} = \frac{X_{t_+}^{-\infty}}{t_+} \frac{t_+}{t_-}.
\end{align}
Here, $\lim_{t \to \infty} {X_{t_-}^{-\infty}}/{t_-} = \lim_{t \to \infty} {X_{t_+}^{-\infty}}/{t_+} = {1}/(\E[N_0] \E[w_0])$
almost surely by \eqref{eq:limrho_n/R_n}.
It thus remains to show that $\lim_{t \to \infty} t_+/t_- = 1$ almost surely or, equivalently,
$\lim_{n \to \infty} \varrho^{-\infty}_{n+1}/\varrho_n^{-\infty} = 1$ almost surely.
The latter, however, follows from \eqref{eq:limrho_n/R_n} in combination with $\E[w_0] < \infty$ and $R_{n+1}/R_n \to 1$ almost surely.
\end{proof}


%
\bibliographystyle{abbrv}

\bibliography{Bouchaud} 

\begin{thebibliography}{10}

\bibitem{Aidekon:2014}
E.~A\"{\i}d\'{e}kon.
\newblock Speed of the biased random walk on a {G}alton-{W}atson tree.
\newblock {\em Probab. Theory Related Fields}, 159(3-4):597--617, 2014.

\bibitem{A-F+H:2009}
M.~Axelson-Fisk and O.~H\"{a}ggstr\"{o}m.
\newblock Biased random walk in a one-dimensional percolation model.
\newblock {\em Stochastic Process. Appl.}, 119(10):3395--3415, 2009.

\bibitem{BenArous+Fribergh:2016}
G.~Ben~Arous and A.~Fribergh.
\newblock Biased random walks on random graphs.
\newblock In {\em Probability and statistical physics in {S}t. {P}etersburg},
  volume~91 of {\em Proc. Sympos. Pure Math.}, pages 99--153. Amer. Math. Soc.,
  Providence, RI, 2016.

\bibitem{Berger+Gantert+Peres:2003}
N.~Berger, N.~Gantert, and Y.~Peres.
\newblock The speed of biased random walk on percolation clusters.
\newblock {\em Probab. Theory Related Fields}, 126(2):221--242, 2003.

\bibitem{Billingsley:1999}
P.~Billingsley.
\newblock {\em Convergence of probability measures}.
\newblock Wiley Series in Probability and Statistics: Probability and
  Statistics. John Wiley \& Sons, Inc., New York, second edition, 1999.
\newblock A Wiley-Interscience Publication.

\bibitem{Bowditch+Tokushige:2020}
A.~Bowditch and Y.~Tokushige.
\newblock Differentiability of the speed of biased random walks on
  {G}alton-{W}atson trees.
\newblock {\em ALEA Lat. Am. J. Probab. Math. Stat.}, 17(1):609--642, 2020.

\bibitem{Douc+al:2018}
R.~Douc, E.~Moulines, P.~Priouret, and P.~Soulier.
\newblock {\em Markov chains}.
\newblock Springer Series in Operations Research and Financial Engineering.
  Springer, Cham, 2018.

\bibitem{Fribergh+Hammond:2014}
A.~Fribergh and A.~Hammond.
\newblock Phase transition for the speed of the biased random walk on the
  supercritical percolation cluster.
\newblock {\em Comm. Pure Appl. Math.}, 67(2):173--245, 2014.

\bibitem{Gantert+Klenke:2022}
N.~Gantert and A.~Klenke.
\newblock The tail of the length of an excursion in a trap of random size.
\newblock {\em J. Stat. Phys.}, 188(3):Paper No. 27, 20, 2022.

\bibitem{Gantert+al:2018}
N.~Gantert, M.~Meiners, and S.~M\"{u}ller.
\newblock Regularity of the speed of biased random walk in a one-dimensional
  percolation model.
\newblock {\em J. Stat. Phys.}, 170(6):1123--1160, 2018.

\bibitem{Luebbers+Meiners:2019}
J.-E. L\"{u}bbers and M.~Meiners.
\newblock The speed of critically biased random walk in a one-dimensional
  percolation model.
\newblock {\em Electron. J. Probab.}, 24:Paper No. 23, 29, 2019.

\bibitem{Lyons+Pemantle+Peres:1996}
R.~Lyons, R.~Pemantle, and Y.~Peres.
\newblock Biased random walks on {G}alton-{W}atson trees.
\newblock {\em Probab. Theory Related Fields}, 106(2):249--264, 1996.

\bibitem{Sznitman:2003}
A.-S. Sznitman.
\newblock On the anisotropic walk on the supercritical percolation cluster.
\newblock {\em Comm. Math. Phys.}, 240(1-2):123--148, 2003.

\end{thebibliography}


\end{document}